\documentclass[reqno, 12pt]{amsart}
\usepackage{amsmath,amsxtra,amssymb,latexsym, amscd,amsthm}
\usepackage[utf8]{inputenc}
\usepackage{enumerate}
\usepackage[mathscr]{euscript}
\usepackage{mathrsfs}
\usepackage[english]{babel}
\usepackage{color}
\usepackage{cite}
\newtheorem {theorem}{Theorem}[section]

\newtheorem {lemma}[theorem]{Lemma}
\newtheorem {corollary}[theorem]{Corollary}
\newtheorem {definition}[theorem]{Definition}
\newtheorem {example}[theorem]{Example}
\newtheorem {remark}[theorem]{Remark}

\def\ar{a\kern-.370em\raise.16ex\hbox{\char95\kern-0.53ex\char'47}\kern.05em}
\def\ees{{\accent"5E e}\kern-.385em\raise.2ex\hbox{\char'23}\kern-.08em}
\def\eex{{\accent"5E e}\kern-.470em\raise.3ex\hbox{\char'176}}
\def\AR{A\kern-.46em\raise.80ex\hbox{\char95\kern-0.53ex\char'47}\kern.13em}
\def\EES{{\accent"5E E}\kern-.5em\raise.8ex\hbox{\char'23 }}
\def\EEX{{\accent"5E E}\kern-.60em\raise.9ex\hbox{\char'176}\kern.1em}
\def\ow{o\kern-.42em\raise.82ex\hbox{
\vrule width .12em height .0ex depth .075ex \kern-0.16em \char'56}\kern-.07em}
\def\OW{O\kern-.460em\raise1.36ex\hbox{
\vrule width .13em height .0ex depth .075ex \kern-0.16em \char'56}\kern-.07em}
\def\UW{U\kern-.42em\raise1.36ex\hbox{
\vrule width .13em height .0ex depth .075ex \kern-0.16em \char'56}\kern-.07em}
\def\DD{D\kern-.7em\raise0.4ex\hbox{\char '55}\kern.33em}

\title[]{The local minimality of differentiable functions}

\author{TI\EES N-S\OW N PH\d{A}M}
\address{Department of Mathematics, Dalat University, 1 Phu Dong Thien Vuong, Dalat, Vietnam}
\email{sonpt@dlu.edu.vn}

\date{ \today}



\begin{document}

\begin{abstract} 
In this paper we present necessary and sufficient conditions (in terms of {\L}ojasiewicz inequalities) 
for the stability of local minimum points in smooth unconstrained optimization.
In particular, we derive a sufficient condition for which  the local minimum property of a given function is determined by its Taylor polynomial of a certain degree.
\end{abstract}

\maketitle

\section{Introduction}
We are interested in the question under which conditions the local minimum property of a differentiable function is determined by its Taylor polynomial of a certain degree. 
More precisely, let $f \colon \mathbb{R}^n \to \mathbb{R}$ 
be a function of class $C^r$ in a neighbourhood of $\overline{x} \in \mathbb{R}^n,$ and we would like to know whether 
the local minimum property of $f$ at $\overline{x}$ is determined by one of $T^rf(\overline{x}),$ 
where $T^rf(\overline{x})$ stands for the $r$th Taylor polynomial of $f$ at $\overline{x}.$
This question was not considered in optimization, to the best of our knowledge. 
Let us begin with the following examples.

\begin{example}\label{VD11}{\rm
(i) Consider the function $f\colon \mathbb{R} \to \mathbb{R}, x \mapsto x^{r} + R(x),$ where $r$ is a positive integer and $R \colon \mathbb{R} \to \mathbb{R}$ is a $C^\infty$-function defined by
\begin{eqnarray*}
R(x) :=
\begin{cases}
e^{-\frac{1}{x^2}} & \textrm{ if } x \ne 0, \\
0 & \textrm{ otherwise.}
\end{cases}
\end{eqnarray*}
Then the Taylor polynomial $T^rf(0)$ has a local minimum at $0$ if and only if $f$ has a local minimum at $0.$

(ii) Consider the function $f \colon \mathbb{R}^n \to \mathbb{R}, x \mapsto \frac{1}{2} \langle Ax, x\rangle + R(x),$ where $A$ is a nonsingular square matrix of size $n$ and $R$ is a function of class $C^2$ with $T^2R(0) \equiv 0.$ Then 
the Taylor polynomial $T^2f(0)$ has a local minimum at $0$ if and only if $f$ has a local minimum at $0.$ See also Corollary~\ref{HQ3}.
}\end{example}

The analysis of the above examples suggests that the local minimum property of a given function can be determined by its Taylor polynomial of a certain degree.

In this paper we provide necessary and sufficient conditions (in terms of {\L}ojasiewicz inequalities) 
for the stability of local minimum points in smooth unconstrained optimization problems.
In particular, we obtain sufficient conditions for which  the local minimum property of a differentiable function is determined by its Taylor polynomial of a certain degree. The obtained  results, together with the ones in \cite{Guo2019, PHAMTS2020}, can assist in checking local minimum points.

This paper is inspired by works on the sufficiency of jets, that is one of the central themes in the singularity theory of differentiable mappings; see, for example, \cite{Bekka2024, Bochnak1971, Grandjean2004, Kuiper1972, Kuo1969, Kuo1969-2, Kuo1972, Mather1968, Migus2016, Thom1964, Xu2007}. On the other hand, the tools used in this study come from variational analysis and generalized differentiation, see \cite{Mordukhovich2006, Mordukhovich2018, Rockafellar1998}.

Note that, the assumption on the smoothness of the functions in question can be weakened in some situations. 
However, to lighten the exposition, we do not pursue this assumption here.

The rest of the paper is organized as follows. Some definitions and preliminary results from variational analysis are recalled in Section~\ref{Section2}. The results and their proofs are presented in Section~\ref{Section3}. 

\section{Preliminaries} \label{Section2}

\subsection{Notation and definitions} 
Throughout this work we deal with the Euclidean space $\mathbb{R}^n$ equipped with the usual scalar product $\langle \cdot, \cdot \rangle$ and the corresponding norm $\| \cdot\|.$ We denote by $\mathbb{B}_r(x)$ the closed ball centered at $x$ with radius $r;$  when ${x}$ is the origin of $\mathbb{R}^n$ we write $\mathbb{B}_{r}$ instead of $\mathbb{B}_{r}({x}),$ and when $r = 1$ we write  $\mathbb{B}$ instead of $\mathbb{B}_{1}.$ We will adopt the convention that $\inf \emptyset = \infty$ and $\sup \emptyset = -\infty.$

For a nonempty set $\Omega \subset \mathbb{R}^n,$ the closure and interior of $\Omega$ are denoted, respectively, by $\mathrm{cl}\, {\Omega}$ and $\mathrm{int}\, {\Omega},$ while the distance of a point $x \in \mathbb{R}^n$ from the set $\Omega$ is defined by
\begin{eqnarray*}
\mathrm{dist}(x, \Omega) &:=& \inf_{y \in \Omega} \|x - y\|.
\end{eqnarray*}
If $\Omega = \emptyset,$ we let $\mathrm{dist}(x, \Omega) := \infty$ for all $x.$ The set $\Omega$ is said to be {\em locally closed} at a point ${x}$ (not necessarily in $\Omega$) if $\Omega \cap V$ is closed for some closed neighbourhood $V$ of ${x}$ in $\mathbb{R}^n.$ 

Let $\overline{\mathbb{R}} := \mathbb{R} \cup \{\infty\}.$ For an extended real valued function $f \colon \mathbb{R}^n \rightarrow \overline{\mathbb{R}},$ we denote its {\em effective domain} by
\begin{eqnarray*}
\mathrm{dom} f &:=& \{ x \in \mathbb{R}^n \ | \ f(x) < \infty  \}.
\end{eqnarray*}
We call $f$ a {\em proper} function if $f(x) < \infty$ for at least one $x \in \mathbb{R}^n,$ or in other words, if $\mathrm{dom} f$ is a nonempty set.
The function $f$ is said to be {\em lower semicontinuous} if for each $x \in \mathbb{R}^n$ the inequality $\liminf_{x' \to {x}} f(x') \geqslant f({x})$ holds.

\subsection{Subdifferentials}

Here we recall some definitions and properties of subdifferentials of real-valued functions, which can be found in~\cite{Mordukhovich2006, Mordukhovich2018, Rockafellar1998}.

\begin{definition}[subdifferentials]{\rm
(i) The {\em Fr\'echet subdifferential} $\hat{\partial} f(x)$ of a lower semicontinuous function $f$ at $x \in \mathbb{R}^n$ is given by
\begin{eqnarray*}
\widehat{\partial} f(x) &:=& \left \{ v \in {\Bbb R}^n \ | \ \liminf_{y \to x, y \ne x} \frac{f(y) - f(x) - \langle v, y - x \rangle}{\|y - x\|} \ge 0 \right \}
\end{eqnarray*}
whenever $x \in \mathrm{dom} f,$ and by $\widehat{\partial}  f({x}) := \emptyset$ otherwise.

(ii) The {\em limiting subdifferential} at $x \in U,$ denoted by ${\partial} f(x),$ is the set of all cluster
points of sequences $\{v_k\}_{k \ge 1}$ such that $v_k \in \widehat{\partial} f(x_k)$ and $(x_k, f(x_k)) \to (x, f(x))$ as $k \to \infty.$
}\end{definition}

\begin{remark}{\rm
(i) If the function $f$ is of class $C^1,$ the above notion coincides with the usual concept of gradient; that is, $\widehat{\partial} f(x) = {\partial} f(x) = \{\nabla f(x) \}.$

(ii) It is a well-known result of variational analysis that $\widehat{\partial} f(x)$ (and a fortiori) $\partial f(x)$ are not empty in a dense subset of the domain of $f.$
}\end{remark}

The following lemmas are well known.

\begin{lemma} \label{Lemma2.7}
For any point $\overline{x} \in \mathbb{R}^n,$ we have
\begin{eqnarray*}
\partial \|\cdot - \overline{x}\| (x) &=& 
\begin{cases}
\mathbb{B} & \quad \textrm{ if } x = \overline{x}, \\
\frac{x - \overline{x}}{\|x - \overline{x}\| } & \quad \textrm{ otherwise.}
\end{cases}
\end{eqnarray*}
\end{lemma}

\begin{lemma}[Fermat's rule] \label{Lemma2.8}
Let $f \colon U \to {\mathbb{R}}$ be a lower semicontinuous function, where $U$ is an open subset of $\mathbb{R}^n.$ If $\overline{x} \in U$ is a local minimum of $f$ then $0 \in \partial f(\overline{x}).$
\end{lemma}

\begin{lemma}\label{Lemma2.9} 
Let $f_i \colon \mathbb{R}^n \to {\mathbb{R}}$, $i=1,\dots,m$ with $m \geqslant 2,$ be lower semicontinuous at $\overline{x} \in \mathbb{R}^n$  and let all but one of these functions be Lipschitz around $\overline{x}.$ Then the following inclusion holds:
\begin{eqnarray*}
\partial \left( f_1+\cdots+f_m\right)(\overline{x}) & \subset & \partial f_1(\overline{x})+\cdots+\partial f_m(\overline{x}).
\end{eqnarray*}
\end{lemma}

\section{The main result and its corollaries} \label{Section3}

Throughout this section, let $f \colon \mathbb{R}^n \to \mathbb{R}$ be a function of class $C^r$ ($r \ge 1$) in a neighbourhood $V$ of a point $\overline{x}$ in $\mathbb{R}^n$ and let $\Sigma \subset \mathbb{R}^n$ be a locally closed set at $\overline{x}$ such that $\{x \in V \mid \nabla f(x) = 0\} \subset \Sigma.$ 

Denote by $T^rf({z})$ the $r$th Taylor polynomial of $f$ at a point ${z} \in V,$ i.e., 
\begin{eqnarray*}
T^rf({z})(x) &:=& \sum_{j = 1}^r \frac{1}{j!} \sum_{i_1 = 1, \ldots, i_j = 1}^n \frac{\partial^j f}{\partial x_{i_1} \cdots \partial x_{i_j}}({z}) \, (x_{i_1 } - {z}_{i_1}) \cdots (x_{i_j} - {z}_{i_j}) 
\end{eqnarray*}
with $x := (x_1, \ldots, x_n) \in  \mathbb{R}^n.$ 

Let $\mathcal{H}_r^{\Sigma}(f, \overline{x}; \overline{w})$ be the {\em horn-neighbourhood} of the level set $f^{-1}(f(\overline{x}))$ relative to $\Sigma$ of degree $r$ and width $\overline{w},$
\begin{eqnarray*}
\mathcal{H}_r^{\Sigma}(f, \overline{x}; \overline{w}) &:=& \{x \in \mathbb{R}^n \mid |f(x) - f(\overline{x})| \le \overline{w}\, \mathrm{dist}(x, \Sigma)^r\}.
\end{eqnarray*}
The original notion of this horn-neighbourhood was introduced in \cite{Kuo1969-2} (see also \cite{Bekka2024}), for $\Sigma = \{\overline{x}\}.$

Let $\mathcal{F}_r^{\Sigma}(\overline{x})$ denote the set of all functions $g  \colon \mathbb{R}^n \to \mathbb{R}$ 
of class $C^r$ in some neighbourhood $U$ of $\overline{x}$ such that $T^r g(z) \equiv 0$ for all $z \in \Sigma \cap U.$ 
The following lemma is useful in the sequel.

\begin{lemma}\label{BD31}
For any $g \in \mathcal{F}_r^{\Sigma}(\overline{x}),$ there are positive constants $c$ and $\delta$ such that
\begin{eqnarray*}
|g(x)| &\le& c\, \mathrm{dist}(x, \Sigma)^r \quad \textrm{ and } \quad 
|\nabla g(x)| \ \le \ c\,\mathrm{dist}(x, \Sigma)^{r - 1}
\end{eqnarray*}
for all $x \in \mathbb{B}_\delta(\overline{x}).$
\end{lemma}
\begin{proof}
It is a consequence of the Taylor formula for $C^r$-functions and the assumption on $g.$ See, for example, \cite[Lemma~2.4]{Bekka2024}.
\end{proof}

Inspired by the works in \cite{Bekka2024, Bochnak1971, Kuiper1972, Kuo1969, Kuo1969-2, Kuo1972, Migus2016, Xu2007}, the main result of this paper is as follows.

\begin{theorem}\label{DL1}
If $\overline{x}$ is a local minimum of $f,$ then the following statements are equivalent:
\begin{enumerate}[{\rm (i)}]
\item \label{Item1} For all $g \in \mathcal{F}_r^{\Sigma}(\overline{x}),$ $\overline{x}$ is a local minimum of $f + g.$

\item \label{Item2} There are positive constants $c$ and $\delta$ such that
\begin{eqnarray*}
f(x) - f(\overline{x}) &\ge& c\, \mathrm{dist}(x, \Sigma)^{r} \quad \textrm{ for all } \quad  x \in \mathbb{B}_\delta(\overline{x}).
\end{eqnarray*}

\item \label{Item3} There are positive constants $c$ and $\delta$ such that
\begin{eqnarray*}
\|\nabla f(x)\| &\ge& c\, \mathrm{dist}(x, \Sigma)^{r - 1} \quad \textrm{ for all } \quad  x \in \mathbb{B}_\delta(\overline{x}).
\end{eqnarray*}

\item \label{Item4} There are positive constants $c$ and $\delta$ such that
\begin{eqnarray*}
\mathrm{dist}(x, \Sigma) \|\nabla f(x)\| + |f(x) - f(\overline{x})| &\ge& c\, \mathrm{dist}(x, \Sigma)^{r} \quad \textrm{ for all } \quad  x \in \mathbb{B}_\delta(\overline{x}).
\end{eqnarray*}

\item \label{Item5} There are positive constants $c, \delta$ and $\overline{w}$ such that
\begin{eqnarray*}
\|\nabla f(x)\| &\ge& c\, \mathrm{dist}(x, \Sigma)^{r - 1} \quad \textrm{ for all } \quad  x \in \mathcal{H}_r^{\Sigma}(f, \overline{x}; \overline{w}) \cap \mathbb{B}_\delta(\overline{x}).
\end{eqnarray*}
\end{enumerate}

\end{theorem}

\begin{proof}
The implications \eqref{Item3} $\Rightarrow$ \eqref{Item4} $\Rightarrow$ \eqref{Item5} are obvious.

\eqref{Item1} $\Rightarrow$ \eqref{Item3}: 
Suppose on the contrary that there exists a sequence $x_k \in \mathbb{R}^n$ tending to $\overline{x}$ such that 
\begin{eqnarray}\label{Eqn1}
\|\nabla f({x}_k) \| & < & \frac{1}{k} \mathrm{dist}({x}_k, \Sigma)^{r - 1}.
\end{eqnarray}
Certainly $x_k \not \in \Sigma,$ and so $f(x_k) > f(\overline{x})$ for $k$ sufficiently large because otherwise $x_k$ is a local minimum point of $f,$ which yields $\nabla f(x_k) = 0$ (by Fermat's rule) and so $x_k \in \Sigma$ (by assumption) a contradiction.

Passing to a subsequence if necessary, we may assume that 
\begin{eqnarray*}
\mathrm{dist}({x}_{k + 1}, \Sigma)^{r - 1}  & < & \frac{1}{2} \mathrm{dist}({x}_k, \Sigma)^{r - 1}.
\end{eqnarray*}
Choose $0 < \delta_k \le \frac{1}{4} \mathrm{dist}({x}_k, \Sigma)^{r - 1}.$ Then $\mathbb{B}_{\delta_k}(x_k)$ is a family of pairwise disjoint closed balls. Shrinking $\delta_k$ if necessary, from \eqref{Eqn1} we can assume that
\begin{eqnarray} \label{Eqn2}
\|\nabla f({x}) \| & < & \frac{2}{k} \mathrm{dist}({x}, \Sigma)^{r - 1} \quad \textrm{ for all } \quad x \in \mathbb{B}_{\delta_k}(x_k).
\end{eqnarray}
Let $\alpha \colon \mathbb{R}^n \to \mathbb{R}$ be a $C^\infty$-function such that $\alpha(x) = 0$ for $\|x\| \ge \frac{1}{4}$ and $\alpha(x) = 2$ in some neighbourhood of $0 \in \mathbb{R}^n.$ Define the function $g \colon \mathbb{R}^n \to \mathbb{R}$ of class $C^r$ by 
\begin{eqnarray*}
g(x) & :=& 
\begin{cases}
\alpha\left( \frac{x - x_k}{\delta_k} \right) \big(f(\overline{x}) - f(x) \big) & \textrm{ if } x \in \mathbb{B}_{\delta_k}(x_k) \textrm{ for some } k, \\
0 & \textrm{ otherwise.}
\end{cases}
\end{eqnarray*}
We have 
\begin{eqnarray*}
f(x_k) + g(x_k) &=& f(x_k) + 2 \big(f(\overline{x}) - f(x_k) \big) \ = \ 
2 f(\overline{x}) - f(x_k) \ < \ f(\overline{x})  \ = \ f(\overline{x}) + g(\overline{x}),
\end{eqnarray*}
and so $\overline{x}$ is not a local minimum of $f + g.$ Moreover, by the construction of $g,$ we have $T^rg(x) \equiv 0$ for all $x \in \Sigma \setminus \{\overline{x}\},$ $x$ near $\overline{x}.$ Therefore, to get a contradiction, it suffices to show that
\begin{eqnarray}\label{Eqn3}
\lim_{x \to \overline{x}} \frac{g(x)}{\|x - \overline{x}\|^r} &=& 0.
\end{eqnarray}
To see this, write $f$ in the form $f = T + R,$ where $T := T^rf(\overline{x})$ is the $r$th Taylor polynomial of $f$ at $\overline{x}$ and $R$ is the remainder. Observe that
\begin{eqnarray} \label{Eqn4}
\lim_{x \to \overline{x}} \frac{R(x)}{\|x - \overline{x}\|^r} &=& 0 \quad \textrm{ and } \quad 
\lim_{x \to \overline{x}} \frac{\|\nabla R(x)\|}{\|x - \overline{x}\|^{r - 1}} \ = \ 0.
\end{eqnarray}
On the other hand, by the Bochnak--{\L}ojasiewicz inequality (see \cite{Bochnak1971}), there exists $c_1 > 0$ such that for all $x$ near $\overline{x}$ we have 
\begin{eqnarray*}
|T(x) - T(\overline{x})| &\le& c_1 \|x - \overline{x}\|\|\nabla T(x)\|.
\end{eqnarray*}
Hence
\begin{eqnarray*}
|f(x) - f(\overline{x})| &\le& 
|T(x) - T(\overline{x})| + |R(x)| \\
&\le&  c_1 \|x - \overline{x}\|\|\nabla T(x)\| + |R(x)| \\
& \le & c_1 \|x - \overline{x}\|\|\nabla f(x)\| + c_1 \|x- \overline{x}\|\|\nabla R(x)\| + |R(x)|.
\end{eqnarray*}
This, together with \eqref{Eqn2}, \eqref{Eqn4} and the boundedness of $\alpha,$ gives the desired equality~\eqref{Eqn3}.

\eqref{Item2} $\Rightarrow$ \eqref{Item1}: 
Assume that there are constants $c > 0$ and $\delta > 0$ such that
\begin{eqnarray*}
f(x) - f(\overline{x}) &\ge& c\, \mathrm{dist}(x, \Sigma)^{r} \quad \textrm{ for all } \quad  x \in \mathbb{B}_\delta(\overline{x}).
\end{eqnarray*}
Take any $g \in \mathcal{F}_r^{\Sigma}(\overline{x}).$ Then applying Lemma~\ref{BD31} (and shrinking $\delta$ if necessary), we have for all $x \in \mathbb{B}_\delta(\overline{x})$ that
\begin{eqnarray*}
|g(x)| &<& \frac{c}{2} \mathrm{dist}({x}, \Sigma)^{r},
\end{eqnarray*}
which yields
\begin{eqnarray*}
f(x) + g(x)  &\ge& f(\overline{x}) + \frac{c}{2} \mathrm{dist}({x}, \Sigma)^{r} \ \ge \ f(\overline{x}) \ = \ f(\overline{x}) + g(\overline{x}).
\end{eqnarray*}
Consequently, $\overline{x}$ is a local minimum of $f + g,$ as required.

\eqref{Item3} $\Rightarrow$ \eqref{Item2}: 
By contradiction, assume that there exists a sequence $x_k \in \mathbb{R}^n$ tending to $\overline{x}$ such that
\begin{eqnarray*}
f(x_k) - f(\overline{x}) &<& \frac{1}{k} \mathrm{dist}(x_k, \Sigma)^r.
\end{eqnarray*}
Reducing $\delta$ if necessary, we may assume that $\overline{x}$ is a minimizer of $f$ over the closed ball $\mathbb{B}_\delta(\overline{x}).$
Then for all $k$ sufficiently large, we have $x_k \in \mathbb{B}_\delta(\overline{x})$ and
\begin{eqnarray*}
f(\overline{x})  &=& \min_{\mathbb{B}_\delta(\overline{x})} f(x) \ \le \ f(x_k) \ < \ f(\overline{x}) + \frac{1}{k} \mathrm{dist}(x_k, \Sigma)^r.
\end{eqnarray*}
Consequently, $\mathrm{dist}(x_k, \Sigma) > 0.$

Applying the Ekeland variational principle (see \cite[Theorem~1.1]{Ekeland1974}) to the function $f$ and the complete metric space
$\mathbb{B}_\delta(\overline{x})$ with $\epsilon := \frac{1}{k} \mathrm{dist}(x_k, \Sigma)^r > 0, \lambda := \frac{1}{2} \mathrm{dist}(x_k, \Sigma) > 0,$ and the initial point $x_k,$ we find ${x}'_k \in \mathbb{B}_\delta(\overline{x})$ such that the following conditions hold:
\begin{enumerate}[{\rm (a)}]
\item $\|{x}'_k - x_k\|  \le \lambda,$
\item $f(\overline{x}) \le f({x}'_k) \le f({x}_k),$ and
\item $f(x) + \frac{\epsilon}{\lambda} \|x - {x}'_k\| \ge f({x}'_k)$ for all $x \in \mathbb{B}_\delta(\overline{x}).$
\end{enumerate}

We have
\begin{eqnarray*}
\|{x}'_k - \overline{x}\| &\le& 
\|{x}'_k - x_k\| + \|x_k - \overline{x}\| \ \le \ \lambda + \|x_k - \overline{x}\| \ \le \ \frac{3}{2} \|x_k - \overline{x}\|. 
\end{eqnarray*}
Consequently, for all $k$ sufficiently large we have ${x}'_k \in \mathrm{int} \mathbb{B}_\delta(\overline{x}).$ This, together with the condition (c), implies that ${x}'_k$ is a minimizer of the function  $f(x) + \frac{\epsilon}{\lambda} \|x - {x}'_k\|$ on the open ball $ \mathrm{int} \mathbb{B}_\delta(\overline{x}).$ By the Fermat rule (Lemma~\ref{Lemma2.8}), we get
\begin{eqnarray*}
0 &\in& \partial \left( f + \frac{\epsilon}{\lambda}\|\cdot - {x}'_k\|\right)({x}'_k).
\end{eqnarray*}
Since the function $\|\cdot-{x}'_k\|$ is 1-Lipschitz, the sum rule (see Lemma~\ref{Lemma2.9}) gives
\begin{eqnarray*}
0 &\in& \nabla f({x}'_k) + \frac{\epsilon}{\lambda} \partial (\|\cdot-{x}'_k\|)({x}'_k).
\end{eqnarray*}
Note that $\partial (\|\cdot-{x}'_k\|)({x}'_k)=\mathbb{B}$ (by Lemma~\ref{Lemma2.7}). Hence  $0 \in \nabla f({x}'_k)  + \frac{\epsilon}{\lambda}\mathbb{B},$ and so
\begin{eqnarray*}
\|\nabla f({x}'_k) \| & \leq & \frac{\epsilon}{\lambda} \ = \ \frac{2}{k} \frac{\mathrm{dist}(x_k, \Sigma)^r}{\mathrm{dist}(x_k, \Sigma)}
\ = \ \frac{2}{k} {\mathrm{dist}(x_k, \Sigma)^{r - 1}}.
\end{eqnarray*}
On the other hand, it follows from the condition (a) that 
\begin{eqnarray*}
\frac{1}{2} \mathrm{dist}(x_k, \Sigma) &\le& \mathrm{dist}({x}'_k, \Sigma) \ \le \ \frac{3}{2} \mathrm{dist}(x_k, \Sigma).
\end{eqnarray*}
Therefore,
\begin{eqnarray*}
\|\nabla f({x}'_k) \| & \leq &  \frac{2^r}{k} \mathrm{dist}({x}'_k, \Sigma)^{r - 1},
\end{eqnarray*}
which contradicts our assumption.

\eqref{Item5} $\Rightarrow$ \eqref{Item1}: 
Assume that there are positive constants $c, \delta$ and $\overline{w}$ such that
\begin{eqnarray*}
\|\nabla f(x)\| &\ge& c\, \mathrm{dist}(x, \Sigma)^{r - 1} \quad \textrm{ for all } \quad  x \in \mathcal{H}_r^{\Sigma}(f, \overline{x}; \overline{w}) \cap \mathbb{B}_\delta(\overline{x}).
\end{eqnarray*}

By contradiction, suppose that for some $g \in \mathcal{F}_r^{\Sigma}(\overline{x}),$ $\overline{x}$ is not a local minimum of $f + g.$ Then there exists a sequence $x_k \in \mathbb{R}^n$ tending to $\overline{x}$ such that
\begin{eqnarray*}
f(x_k) + g(x_k)  &<& f(\overline{x}) + g(\overline{x})  \ = \ f(\overline{x}).
\end{eqnarray*}
Reducing $\delta$ if necessary, we may assume that $\overline{x}$ is a minimizer of $f$ over the closed ball $\mathbb{B}_\delta(\overline{x}).$
Then for all $k$ sufficiently large, we have $x_k \in \mathbb{B}_\delta(\overline{x})$ and
\begin{eqnarray*}
f(\overline{x})  &=& \min_{\mathbb{B}_\delta(\overline{x})} f(x) \ \le \ f(x_k) \ < \ f(\overline{x}) - g(x_k).
\end{eqnarray*}
Consequently, $-g(x_k) > 0,$ and so $x_k \not \in \Sigma.$

Applying the Ekeland variational principle (see \cite[Theorem~1.1]{Ekeland1974}) to the function $f$ and the complete metric space
$\mathbb{B}_\delta(\overline{x})$ with $\epsilon := -g(x_k) > 0, \lambda := \frac{1}{2} \mathrm{dist}(x_k, \Sigma) > 0,$ and the initial point $x_k,$ we find ${x}'_k \in \mathbb{B}_\delta(\overline{x})$ such that
\begin{enumerate}[{\rm (a)}]
\item $\|{x}'_k - x_k\|  \le \lambda,$
\item $f(\overline{x}) \le f({x}'_k) \le f({x}_k),$ and
\item ${x}'_k$ is a minimizer of the function $f(x) + \frac{\epsilon}{\lambda} \|x - {x}'_k\|$ over $\mathbb{B}_\delta(\overline{x}).$
\end{enumerate}

We have
\begin{eqnarray*}
\|{x}'_k - \overline{x}\| &\le& 
\|{x}'_k - x_k\| + \|x_k - \overline{x}\| \ \le \ \lambda + \|x_k - \overline{x}\| \ \le \ \frac{3}{2} \|x_k - \overline{x}\|. 
\end{eqnarray*}
Consequently, we have for all $k$ sufficiently large that ${x}'_k \in \mathrm{int} \mathbb{B}_\delta(\overline{x}),$ which, together with the condition~(c), yields
\begin{eqnarray} \label{Eqn5}
f(x) + \frac{\epsilon}{\lambda} \|x - {x}'_k\| &\ge& f({x}'_k) \quad \textrm{ for all } \quad x \in \mathrm{int} \mathbb{B}_\delta(\overline{x}).
\end{eqnarray}

The condition (a) implies easily that 
\begin{eqnarray*}
\frac{1}{2} \mathrm{dist}(x_k, \Sigma) &\le& \mathrm{dist}({x}'_k, \Sigma) \ \le \ \frac{3}{2} \mathrm{dist}(x_k, \Sigma).
\end{eqnarray*}
Take any $0 < c' < \frac{1}{2^r} \min \{c, \overline{w}\}.$ Since $g \in \mathcal{F}_r^{\Sigma}(\overline{x}),$ we have for all $k$ sufficiently large
that (see Lemma~\ref{BD31})
\begin{eqnarray*}
|g(x_k)| &\le& c' \mathrm{dist}(x_k, \Sigma)^{r} \ \le \ 2^r c' \, \mathrm{dist}({x}'_k, \Sigma)^{r} \ \le \ \overline{w}\, \mathrm{dist}({x}'_k, \Sigma)^{r}.
\end{eqnarray*}
This, together with the condition (b), implies that
\begin{eqnarray*}
0 &\le& f({x}'_k) - f(\overline{x}) \ \le \ f(x_k) - f(\overline{x}) \ < \ -g(x_k) \ \le \ \overline{w}\, \mathrm{dist}({x}'_k, \Sigma)^{r},
\end{eqnarray*}
which yields ${x}'_k \in \mathcal{H}_r^{\Sigma}(f, \overline{x}; \overline{w}).$

The condition~\eqref{Eqn5}, together with the Fermat rule (Lemma~\ref{Lemma2.8}), implies that
\begin{eqnarray*}
0 &\in& \partial \left( f + \frac{\epsilon}{\lambda}\|\cdot - {x}'_k \|\right)({x}'_k).
\end{eqnarray*}
Since the function $\|\cdot - {x}'_k\|$ is 1-Lipschitz, the sum rule (see Lemma~\ref{Lemma2.9}) gives
\begin{eqnarray*}
0 &\in& \nabla f({x}'_k) + \frac{\epsilon}{\lambda} \partial (\|\cdot - {x}'_k\|)({x}'_k).
\end{eqnarray*}
Note that $\partial (\|\cdot-{x}'_k\|)({x}'_k)=\mathbb{B}$ (by Lemma~\ref{Lemma2.7}). Hence $0 \in \nabla f({x}'_k)  + \frac{\epsilon}{\lambda}\mathbb{B},$ and so
\begin{eqnarray*}
\|\nabla f({x}'_k) \| & \leq & \frac{\epsilon}{\lambda} \ = \ \frac{-2g(x_k)}{\mathrm{dist}(x_k, \Sigma)} \ \le \ 
 2 c'\, \mathrm{dist}(x_k, \Sigma)^{r - 1} \ < \  c\, \mathrm{dist}({x}'_k, \Sigma)^{r - 1},
\end{eqnarray*}
which contradicts our assumption.
\end{proof}

\begin{remark}{\rm
(i) In the general case (i.e., $\overline{x}$ is not necessary  a local minimum of $f$), we always have 
\eqref{Item3} $\Rightarrow$ \eqref{Item4} $\Leftrightarrow$ \eqref{Item5} but not
\eqref{Item4} $\Rightarrow$ \eqref{Item3},  see \cite[Example~4.6]{Bekka2024}.

(ii) By the works in \cite{Bekka2024, Kuiper1972, Kuo1969, Migus2016, Xu2007}, we know that \eqref{Item3} is equivalent to $\Sigma$-$C^0$-sufficiency of relative $r$-jets while \eqref{Item4} (and hence \eqref{Item5}) is equivalent to $\Sigma$-$V$-sufficiency of relative $r$-jets. These results, together with Theorem~\ref{DL1}, shows that the latter notions are also equivalent in the case where $\overline{x}$ is a local minimum of $f.$ Since we do not use this fact, we leave the details for the interested reader. 

(iii) If the function $f$ is analytic  in a neighbourhood of $\overline{x},$ then there are positive constants $c,$ $\delta$ and $\alpha$ satisfying the 
{\L}ojasiewicz inequality
\begin{eqnarray*}
\|\nabla f(x)\| &\ge& c\, \mathrm{dist}(x, \Sigma)^{\alpha} \quad \textrm{ for all } \quad  x \in \mathbb{B}_\delta(\overline{x})
\end{eqnarray*}
(see \cite{Bierstone1988}), and hence Item~\eqref{Item3} in Theorem~\ref{DL1} holds for $r \ge \alpha + 1.$
}\end{remark}

The following corollary asserts that the omission of terms of degree $\ge r + 1$ would not change the local minimum property of $f$ and that small enough perturbations even at degree $r$ would not either. Recall that $T^r f(z)$ stands for the $r$th Taylor polynomial of $f$ at $z.$

\begin{corollary}\label{HQ1}
Let $c, \delta$ and $\overline{w}$ be positive constants satisfying the {\L}ojasiewicz inequality
\begin{eqnarray*}
\|\nabla f(x)\| &\ge& c\, \mathrm{dist}(x, \Sigma)^{r - 1} \quad \textrm{ for all } \quad  x \in \mathcal{H}_r^{\Sigma}(f, \overline{x}; \overline{w}) \cap \mathbb{B}_\delta(\overline{x}).
\end{eqnarray*}
If in addition $T^rf(\cdot)$ is constant on $\Sigma \cap \mathbb{B}_\delta(\overline{x}),$ then the following hold:
\begin{enumerate}[{\rm (i)}]
\item $f$ has a local minimum at $\overline{x}$ if and only if $T^r f(\overline{x})$ has a local minimum at $\overline{x}.$ 

\item Assume that $f$ has a local minimum at $\overline{x}$ and let ${h} \colon \mathbb{R}^n \to \mathbb{R}$ be a function such that
\begin{eqnarray*}
|{h}(x) - {h}(\overline{x})| &\le& \epsilon \, \mathrm{dist}(x, \Sigma)^{r},
\end{eqnarray*}
where $\epsilon > 0$ sufficiently small, $x$ near $\overline{x}.$ Then $f + {h}$ has a local minimum at $\overline{x}.$ 
\end{enumerate}
\end{corollary}

\begin{proof}
(i) Let $T := T^r f(\overline{x})$ and $g  := T - f.$ We have $g \in \mathcal{F}_r^{\Sigma}(\overline{x})$ by assumption.
Then the necessary condition follows directly from Theorem~\ref{DL1}.

Conversely, assume that $\overline{x}$ is a local minimum of the polynomial $T.$ Since $g \in \mathcal{F}_r^{\Sigma}(\overline{x}),$ we have for all $x \in \mathbb{B}_\delta(\overline{x})$ that
\begin{eqnarray*}
|g(x)| &\le& \frac{\overline{w}}{2} \, \mathrm{dist}(x, \Sigma)^r \quad \textrm{ and } \quad 
|\nabla g(x)| \ \le \ \frac{c}{2} \,\mathrm{dist}(x, \Sigma)^{r - 1}
\end{eqnarray*}
(perhaps after reducing $\delta$; see Lemma~\ref{BD31}). If in addition $x \in \mathcal{H}_r^{\Sigma}(T, \overline{x}; \frac{\overline{w}}{2}),$ then
\begin{eqnarray*}
|f(x) - f (\overline{x})| &\le& 
|T(x) - T (\overline{x})| + |g(x)| \ \le \ \overline{w} \, \mathrm{dist}(x, \Sigma)^r,
\end{eqnarray*}
which, together with our assumption, yields
\begin{eqnarray*}
\|\nabla T(x)\| &\ge& \|\nabla f(x)\| - \|\nabla g(x)\|  \ \ge \  \frac{c}{2}\, \mathrm{dist}(x, \Sigma)^{r - 1}.
\end{eqnarray*}
Repeating the proof of the implication~\eqref{Item5} $\Rightarrow$ \eqref{Item1} of Theorem~\ref{DL1} 
with $f$ replaced by $T$ and $g$ replaced by $-g,$ we get that $\overline{x}$ is a local minimum of $T + (-g) = f.$
(Note that we do not assume here that $\{x \in \mathbb{B}_\delta(\overline{x})  \mid \nabla T(x) = 0\} \subset \Sigma.$)

(ii) Indeed, by Theorem~\ref{DL1}, there are positive constants $c'$ and $\delta'$ such that
\begin{eqnarray*}
f(x) - f(\overline{x}) &\ge& c'\, \mathrm{dist}(x, \Sigma)^{r} \quad \textrm{ for all } \quad  x \in \mathbb{B}_{\delta'}(\overline{x}).
\end{eqnarray*}
Let ${h} \colon \mathbb{R}^n \to \mathbb{R}$ be a function such that
\begin{eqnarray*}
|{h}(x) - {h}(\overline{x})| &\le& \frac{c'}{2}\, \mathrm{dist}(x, \Sigma)^{r} \quad \textrm{ for all } \quad  x \in \mathbb{B}_{\delta'}(\overline{x}).
\end{eqnarray*}
We have
\begin{eqnarray*}
(f + h)(x) - (f + h)(\overline{x}) &\ge& \frac{c'}{2}\, \mathrm{dist}(x, \Sigma)^{r} \ \ge \ 0 \quad \textrm{ for all } \quad  x \in \mathbb{B}_{\delta'}(\overline{x}),
\end{eqnarray*}
which yields that $f + h$ has a local minimum at $\overline{x}.$
\end{proof}

An important corollary  of Theorem~\ref{DL1} is as follows.

\begin{corollary} \label{HQ2}
Let $c, \delta$ and $\overline{w}$ be positive constants satisfying the {\L}ojasiewicz inequality
\begin{eqnarray*}
\|\nabla f(x)\| &\ge& c\, \|x - \overline{x}\|^{r - 1} \quad \textrm{ for all } \quad  x \in \mathbb{B}_\delta(\overline{x}) \
\textrm{ with } \  |f(x) - f(\overline{x})| \le \overline{w}\, \|x - \overline{x}\|^r.
\end{eqnarray*}
Then the following hold:
\begin{enumerate}[{\rm (i)}]
\item $f$ has a local minimum at $\overline{x}$ if and only if $T^r f(\overline{x})$ has a local minimum at $\overline{x}.$ 

\item Assume that $f$ has a local minimum at $\overline{x}$ and let ${h} \colon \mathbb{R}^n \to \mathbb{R}$ be a function such that
\begin{eqnarray*}
|{h}(x) - {h}(\overline{x})| &\le& \epsilon \, \|x - \overline{x}\|^{r},
\end{eqnarray*}
where $\epsilon > 0$ sufficiently small, $x$ near $\overline{x}.$ Then $f + {h}$ has a local minimum at $\overline{x}.$ 
\end{enumerate}
\end{corollary}

\begin{proof}
Apply Corollary~\ref{HQ1} with $\Sigma = \{\overline{x}\}.$
\end{proof}

Another corollary of Theorem~\ref{DL1} is as follows; see also \cite{Schaffler1992}.
 
\begin{corollary} \label{HQ3}
If $T^{r - 1}f(\overline{x}) \equiv f(\overline{x})$ and $\nabla T^{r}f(\overline{x}) (x) = 0$ if and only if $x = \overline{x},$ then there exist constants $c > 0$ and $\delta > 0$ such that
\begin{eqnarray*}
\|\nabla f(x) \| &\ge& c\, \|x - \overline{x}\|^{r - 1} \quad \textrm{ for all } \quad x \in \mathbb{B}_\delta(\overline{x}).
\end{eqnarray*}
Consequently, the conclusions of Corollary~\ref{HQ2} hold.
\end{corollary}

\begin{proof}
Write $f$ in the form $f = T + R,$ where $T := T^rf(\overline{x})$ is the $r$th Taylor polynomial of $f$ at $\overline{x}$ and $R$ is the remainder.
By assumption, there must exist a constant $c > 0$ such that
\begin{eqnarray*}
\|\nabla T(x) \| &\ge& 2c\, \|x - \overline{x}\|^{r - 1} \quad \textrm{ for all } \quad x \in \mathbb{R}^n.
\end{eqnarray*}
On the other hand, we know that 
$$\lim_{x \to \overline{x}} \frac{\|\nabla R(x)\|}{\|x - \overline{x}\|^{r - 1}} \ = \ 0.$$
Therefore,
\begin{eqnarray*}
\|\nabla f(x) \| &\ge& \|\nabla T(x) \| - \|\nabla R(x)\| \ \ge \ c\, \|x - \overline{x}\|^{r - 1}
\end{eqnarray*}
in some neighbourhood of $\overline{x},$ as required.
\end{proof}

\begin{example}{\rm
(i) Consider the function $f \colon \mathbb{R}^2 \to \mathbb{R}, (x, y) \mapsto x^{r} + x^{r + 1}R(y),$ with $r \ge 1$ and $R$ given by
\begin{eqnarray*}
R(y) :=
\begin{cases}
e^{-\frac{1}{y^2}} & \textrm{ if } y \ne 0, \\
0 & \textrm{ otherwise.}
\end{cases}
\end{eqnarray*}
For $\Sigma := \{x = 0\} \subset  \mathbb{R}^2,$ it is easy to check that there is a constant $c > 0$ such that
\begin{eqnarray*}
\|\nabla f(x, y)\| &\ge& c\, \mathrm{dist}((x, y), \Sigma)^{r - 1}
\end{eqnarray*}
for all $(x, y)$ in a neighbourhood of $(0, 0) \in \mathbb{R}^2.$ Moreover, $T^{r} f(\cdot)$ is constant on $\Sigma.$
In light of Corollary~\ref{HQ1}, $f$ has a local minimum at $(0, 0)$ if and only if the polynomial $T^{r} f(0, 0)$ has a local minimum at $(0, 0).$ 

(ii) Consider the function $f \colon \mathbb{R}^2 \to \mathbb{R}, (x, y) \mapsto x^{3} - 3xy^k + tR(y),$ with $t \in \mathbb{R},$  $k \ge 3$ and $R$ given in the above example. It is not hard to see that (see also \cite[Example~1]{Kuo1969-2}) there exists a constant $c > 0$ such that
\begin{eqnarray*}
\|\nabla f(x, y)\| &\ge& c\, \|(x, y)\|^{\frac{3k}{2} - 1}
\end{eqnarray*}
for all $(x, y)$ in a neighbourhood of $(0, 0) \in \mathbb{R}^2.$ On the other hand, for all $r \ge k + 1$ we have $T^r f(0, 0)$ is the polynomial $x^{3} - 3xy^k,$ and so it has not a local minimum at $(0, 0).$ By Corollary~\ref{HQ2}, $f$ has not a local minimum at $(0, 0).$ 
}\end{example}

\begin{remark}{\rm
There are algorithms for checking whether a given point is a local minimum of a polynomial function, see \cite{Guo2019, PHAMTS2020}. This fact, together with Corollaries~\ref{HQ1}, \ref{HQ2} and \ref{HQ3}, can assist in determining the local minimum property of differentiable functions.
}\end{remark}

From the proof of Theorem~\ref{DL1} we obtain the following stability result.

\begin{corollary}
Let $c, \delta$ and $\overline{w}$ be positive constants satisfying the {\L}ojasiewicz inequality
\begin{eqnarray*}
\|\nabla f(x)\| &\ge& c\, \mathrm{dist}(x, \Sigma)^{r - 1} \quad \textrm{ for all } \quad  x \in \mathcal{H}_r^{\Sigma}(f, \overline{x}; \overline{w}) \cap \mathbb{B}_\delta(\overline{x}),
\end{eqnarray*}
where $r \ge 2$ and $\Sigma := \{x \in V \mid \nabla f(x) = 0\}.$ 
If $f$ has a local minimum at $\overline{x},$ then $f + {h}$ also has a local minimum at $\overline{x}$ for any function ${h} \colon \mathbb{R}^n \to \mathbb{R}$ satisfying
\begin{eqnarray*}
|{h}(x) - {h}(\overline{x})| &\le& \epsilon \, \|\nabla f(x)\|^{r},
\end{eqnarray*}
where $\epsilon > 0$ sufficiently small, $x$ near $\overline{x}.$
\end{corollary}

\begin{proof}
The gradient mapping $\nabla f$ is of class $C^{r - 1},$ and so it is locally Lipschitz. Consequently, there exists a constant $L > 0$ such that for all $x$ near $\overline{x},$ 
\begin{eqnarray*}
\|\nabla f(x)\| &\le& L\, \mathrm{dist}(x, \Sigma).
\end{eqnarray*}
Let $0 < \epsilon < \frac{1}{2^r L^r} \min \{c, \overline{w}\}$ and ${h} \colon \mathbb{R}^n \to \mathbb{R}$ be a function such that for all $x$ near $\overline{x},$ 
\begin{eqnarray*}
|{h}(x) - {h}(\overline{x})| &\le& \epsilon \, \mathrm{dist}(x, \Sigma)^{r}.
\end{eqnarray*}
Then 
\begin{eqnarray*}
|{h}(x) - {h}(\overline{x})| &\le& \epsilon \, \|\nabla f(x)\|^{r} \ \le \ \epsilon\, L^r \mathrm{dist}(x, \Sigma)^{r}.
\end{eqnarray*}
Repeating the proof of the implication \eqref{Item5}~$\Rightarrow$~\eqref{Item1} in Theorem~\ref{DL1} with $g$ replaced by $h - h(\overline{x}),$ we get the desired conclusion.
\end{proof}


\begin{thebibliography}{10}

\bibitem{Bekka2024}
K.~Bekka and S.~Koike.
\newblock Characterisations of {{$V$}}-sufficiency and {{$C^0$}}-sufficiency of
  relative jets.
\newblock {\em Hokkaido Math. J.}, 53(1):1--50, 2024.

\bibitem{Bierstone1988}
E.~Bierstone and P.~D. Milman.
\newblock Semianalytic and subanalytic sets.
\newblock {\em Inst. Hautes \'Etudes Sci. Publ. Math.}, 67:5--42, 1988.

\bibitem{Bochnak1971}
J.~Bochnak and S.~{{\L}}ojasiewicz.
\newblock A converse of the {{K}}uiper--{{K}}uo theorem.
\newblock In {\em Proceedings of Liverpool Singularities Symposium, I
  (1969/70)}, volume 192 of {\em Lecture Notes in Math.}, pages 254--261,
  Springer, Berlin, 1971.

\bibitem{Ekeland1974}
I.~Ekeland.
\newblock On the variational principle.
\newblock {\em J. Math. Anal. Appl.}, 47:324--353, 1974.

\bibitem{Grandjean2004}
V.~Grandjean.
\newblock Infinite relative determinacy of smooth function germs with
  transverse isolated singularities and relative {{\L}}ojasiewicz conditions.
\newblock {\em J. London Math. Soc.}, 69:518--530, 2004.

\bibitem{Guo2019}
F.~Guo and T.~S. Ph\d{a}m.
\newblock On types of degenerate critical points of real polynomial functions.
\newblock {\em J. Symb. Comput.}, 99:108--126, 2020.

\bibitem{Kuiper1972}
N.~H. Kuiper.
\newblock {{$C^1$}}-equivalence of functions near isolated critical points.
\newblock In {\em Symposium on Infnite Dimensional Topology}, volume~69 of {\em
  Ann. of Math. Studies}, Princeton, N.J., 1972. Princeton Univ. Press.

\bibitem{Kuo1969-2}
T.-C. Kuo.
\newblock A complete determination of {{$C^0$}}-sufficiency in {{$J^r(2; 1)$}}.
\newblock {\em Invent. math.}, 8:226--235, 1969.

\bibitem{Kuo1969}
T.~C. Kuo.
\newblock On {{$C^0$}}-sufficiency of jets of potential functions.
\newblock {\em Topology.}, 8:167--171, 1969.

\bibitem{Kuo1972}
T.-C. Kuo.
\newblock Characterizations of {{$V$}}-sufficiency of jets.
\newblock {\em Topology}, 11:115--131, 1972.

\bibitem{Mather1968}
J.~Mather.
\newblock Stability of ${{C}}^1$ mappings {{III}}: Finitely determined
  map-germs.
\newblock {\em Inst. Hautes \'Etudes Sci. Publ. Math.}, 35:279--308, 1968.

\bibitem{Migus2016}
P.~Migus, T.~Rodak, and S.~Spodzieja.
\newblock Finite determinacy of non-isolated singularities.
\newblock {\em Ann. Polon. Math.}, 117:197--206, 2016.

\bibitem{Mordukhovich2006}
B.~S. Mordukhovich.
\newblock {\em Variational Analysis and Generalized Differentiation, I: Basic
  Theory; II: Applications}.
\newblock Springer, Berlin, 2006.

\bibitem{Mordukhovich2018}
B.~S. Mordukhovich.
\newblock {\em Variational Analysis and Applications}.
\newblock Springer, New York, 2018.

\bibitem{PHAMTS2020}
T.~S. Ph\d{a}m.
\newblock Local minimizers of semi-algebraic functions from the viewpoint of
  tangencies.
\newblock {\em SIAM J. Optim.}, 30(3):1777--1794, 2020.

\bibitem{Rockafellar1998}
R.~T. Rockafellar and R.~Wets.
\newblock {\em Variational Analysis}, volume 317 of {\em Grundlehren Math.
  Wiss.}
\newblock Springer, Berlin, 1998.

\bibitem{Schaffler1992}
S.~Sch\"affler.
\newblock Classification of critical stationary points in unconstrained
  optimization.
\newblock {\em SIAM J. Optim.}, 2(1):1--6, 1992.

\bibitem{Thom1964}
R.~Thom.
\newblock Local topological properties of differentiable mappings.
\newblock In {\em Differential Analysis, Bombay Colloq.}, pages 191--202.
  Oxford Univ. Press, London, 1964.

\bibitem{Xu2007}
X.~Xu.
\newblock {{$C^0$}}-sufficiency, {{K}}uiper--{{K}}uo and {{T}}hom conditions
  for non-isolated singularity.
\newblock {\em Acta Mathematica Sinica}, 23:1251--1256, 2007.

\end{thebibliography}
\end{document}